\documentclass[onecolumn,reqno,11pt]{amsart}
\usepackage{amsthm}
\usepackage{times,amssymb,amsmath,amsfonts}

\usepackage{hyperref}

\usepackage{color}
\usepackage{amsmath,amssymb,amsthm}
\usepackage{wrapfig}
\usepackage{ifpdf}

\newtheorem{thm}{Theorem}[section]

\newtheorem{definition}{Definition}[section]
\newtheorem{remark}{Remark}

\title {The kissing number in 48 dimensions for codes with certain forbidden distances
is 52\,416\,000} 
\author{Peter Boyvalenkov, Danila Cherkashin}
\address{ Institute of Mathematics and Informatics, Bulgarian Academy of Sciences,
8 G Bonchev Str., 1113  Sofia, Bulgaria}
\email{peter@math.bas.bg}\email{jiocb@math.bas.bg}
	
\begin{document}

\date{}
\maketitle

\begin{abstract}  
We prove that the kissing number in 48 dimensions among
antipodal spherical codes with certain forbidden inner products is 
52\,416\,000. Constructions of attaining codes as kissing configurations 
of minimum vectors in even unimodular extremal lattices are
well known since the 1970's. We also prove that corresponding spherical
11-designs with the same cardinality are minimal. We use appropriate modifications of
the linear programming bounds for spherical codes and designs introduced
by Delsarte, Goethals and Seidel in 1977. 
\end{abstract}

\section{Introduction}	

Recently, Gon{\c{c}}alves and Vedana~\cite{gonccalves2023sphere} found that under an additional constraint the optimal sphere packing density in dimension 48 is reached if and only if the centers of the packing 
form a 48-dimensional even unimodular extremal lattice. Namely, the additional assumption is that a distance between centers from the packing never belongs to $\left (\frac{4r}{\sqrt{3}}, \frac{2\sqrt{5}r}{\sqrt{3}} \right)$, where $r$ is the radius of a sphere. Up to the date there are known four proper lattices: $P_{48p}, P_{48q}, P_{48m}$ and $P_{48n}$, see~\cite{nebe2014fourth}. All four lattices constitute (via their minimum norm vectors) spherical 
codes of cardinality 52\,416\,000 which are kissing configurations and antipodal spherical 11-designs.  

The result of \cite{gonccalves2023sphere} follows the solutions of the general sphere packing problem in 2017 in 8 dimensions by Viazovska \cite{Via17} and in 24 dimensions by 
Cohn, Kumar, Miller, Radchenko, and Viazovska \cite{CKMRV} who first developed and applied the necessary techniques. 

Now assume that $r = 1$ and consider the kissing number problem: how many non-overlapping unit spheres 
can touch a unit sphere in 48 dimensions. 
Note that the kissing condition means exactly that the spherical code formed by the touching points has no inner product greater than $1/2$ between distinct points.
Translating the above distances to inner products on $\mathbb{S}^{47}$, we 
consider antipodal spherical kissing codes without inner products 
in the set 
\[ F:=\left(-\frac{1}{3},-\frac{1}{6}\right) \cup \left(\frac{1}{6}, \frac{1}{3}\right).\]

Using standard linear programming tools we prove that any antipodal spherical code with no inner products in $F$ cannot have cardinality greater 
than 52\,416\,000 (Theorem~\ref{kiss-48}). This bound is 
achieved by (at least) four non-isomorphic antipodal codes 
mentioned above. Therefore, these four codes are maximal
among the antipodal codes without inner products in $F$. This result follows the only three known kissing numbers in 
dimensions $n \geq 4$; 240 in 8 dimensions and 196\,560 in
24 dimensions, found independently in 1979 by Levenshtein 
\cite{Lev79} and Odlyzko and Sloane \cite{OS79}, and 24 in
4 dimensions found in 2003 by Musin \cite{Mus03}.

We also prove (Theorems~\ref{11des-48} and~\ref{11des-48-2}) 
that spherical 11-designs on $\mathbb{S}^{47}$ without 
inner products either in $F$ or in the set
\[ F^\prime:=\left(-\frac{1}{2},-\frac{1}{3}\right) \cup \left(\frac{1}{3}, \frac{1}{2}\right)\]
cannot have less than 52\,416\,000 points as the attaining examples are the same, featuring now as spherical 11-designs.
Therefore, these four designs are minimal
among the spherical 11-designs without inner products 
either in $F$ or in $F^\prime$. A comprehensive survey on spherical designs can be found in~\cite{bannai2009survey}.

We also compute the distance distribution of the attaining codes (designs)
which appears to be the same for all of them and coincides, of course, with
the distance distribution of the known codes from even unimodular extremal lattices.

Upper bounds on kissing numbers were obtained by linear programming (see \cite[Chapters 9, 13]{CS}, \cite{Lev79,OS79,Lev98,BDM12} and references therein) and, since 2008, by semidefinite programming (see \cite{BV08,LL22} and references therein). We are not aware about bounds for spherical codes with forbidden inner products. The best upper bound for 
unrestricted codes is 867\,897\,072 (taken
from \href{https://cohn.mit.edu/kissing-numbers}{the page of Henry Cohn}).

Lower bounds for the minimum possible size of (unrestricted) spherical designs of given dimension and strength were obtained by Delsarte, Goethals, and Seidel \cite{DGS} with the very definition of spherical designs. Later, some linear programming bounds were found in \cite{Boy95,Y97}. We do not have seen bounds for spherical designs via semidefinite programming. Again, we are not aware about bounds for spherical designs with forbidden inner products (but see \cite{BBD99}). 

The paper is organized as follows. In Section 2 we give some 
notations, definitions and basic facts about Gegenbauer polynomials and spherical designs. Section 3 is devoted to the 
necessary modifications of the classical linear programming
bounds. In Section 4 we describe briefly the origin of the attaining codes and designs. Sections 5 and 6 are devoted to
the derivation of the linear programming bounds for our main 
objects: the kissing number in 48 dimensions with forbidden
inner products in $F$ and the minimal spherical 11-designs 
in 48 dimensions with forbidden inner products either in $F$
or in $F^\prime$.

\section{Notation and preliminaries}

Let $\mathbb{S}^{n-1}$ stand for the unit sphere in $\mathbb {R}^n$ centered at the origin. A \textit{spherical code} is a finite set $C \subset \mathbb{S}^{n-1}$. A spherical code $C$
is called \textit{antipodal}, if $C=-C$.

Given a spherical code $C \subset \mathbb{S}^{n-1}$, denote by 
\[ I(C)=\{ t = \langle x,y \rangle : x,y \in C, x \neq y \} \] 
the set of inner products of distinct points of $C$. If $C$ is antipodal, then $-1 \in I(C)$ 
and the set $I(C) \setminus \{-1\}$ is symmetric 
about the origin.

For any $x \in C$ and $t \in I(C)$, denote by 
\[ A_t(x):=|\{y \in C : x \cdot y = t\}| \]
the number of the points of $C$ with inner product $t$ with
$x$. Then the system of nonnegative integers
$(A_t(x): t \in I(C))$ is called \textit{distance distribution} 
of $C$ with respect to $x$. When the distance distribution 
does not depend on the choice of $x$, the code $C$ is 
called {\it distance invariant} \cite{DGS} and $x$ is omitted in the notation. Note that
$A_{-1}(x) \in \{0,1\}$ and $A_{-1}=1$
means that the code is antipodal. If $C$ is antipodal, then $A_t(x)=A_{-t}(x)$ for every 
$t \in I(C) \setminus \{-1\}$ and every $x \in C$.

Define the sequence of Gegenbauer polynomials $\{P_i^{(n)}\}_{i=0}^\infty$ by $P_0^{(n)}(t) = 1$, $P_1^{(n)}(t) = t$,
\[
P_i^{(n)}(t) = \frac{(n + 2i - 4) t P_{i-1}^{(n)}(t) - (i-1) P_{i-2}^{(n)}(t)}{n + i - 3}
\]
for $i \geq 2$ (cf. \cite{Szego1975}). Clearly $P_i^{(n)}(1) = 1$ for every $i$. Such polynomials are orthogonal with the weight $w(t) = (1-t^2)^{(n-3)/2}$, i.e.
\[
\int_{-1}^{1} P_i^{(n)}(t) \cdot P_j^{(n)}(t) \cdot w(t) dt = 0,
\]
for $i \neq j$. Thus every polynomial $f$ has a unique representation 
\[ f(t)=\sum_{i=0}^{\deg{f}} f_iP_i^{(n)}(t), \] where
\begin{equation} \label{eq:orthorep}
f_i = \frac{\int_{-1}^{1} f(t) \cdot P_i^{(n)}(t) \cdot w(t) dt}{\int_{-1}^{1} \left(P_i^{(n)}(t) \right)^2 \cdot w(t) dt}.    
\end{equation}

A crucial property of Gegenbauer polynomials in the derivation of linear programming bounds is their positive definiteness, in the sense that
\[
M_i (C) := \sum_{x,y \in C} P_i^{(n)} (\langle x,y \rangle) \geq 0
\]
for every code $C$ and every positive integer $i$ \cite{Sch42}. We call $M_i(C)$ the \textit{$i$-th moment} of the code $C$. It is clear that
$M_i(C)=0$ for odd $i$ and antipodal $C$.

\begin{definition}
A spherical $\tau$-design is a spherical 
code $C \subset \mathbb{S}^{n-1}$ such that
\[ \int_{\mathbb{S}^{n-1}} p(x) d\sigma_n(x)= \frac{1}{|C|} \sum_{x \in C} p(x) \]
($\sigma_n $ is the normalized surface measure) holds for all polynomials $p(x) = p(x_1,x_2,\ldots,x_n)$ of degree at most $\tau$. 
\end{definition}

Equivalently, $C \subset \mathbb{S}^{n-1}$ is a spherical $\tau$-design if 
\begin{equation} \label{def-ft}
    \sum_{y \in C} f( \langle x, y \rangle) = f_0 \cdot |C|
\end{equation}
for every point $x \in \mathbb{S}^{n-1}$ and every polynomial $f$ of degree at most $\tau$, where $f_0$ is defined by~\eqref{eq:orthorep}, see, for example, \cite[Definition 1.10]{FL}.
Another equivalent way to define a $\tau$-design $C$ is to demand the equalities $M_1(C) = M_2(C) = \dots = M_\tau(C) = 0$.

We will use \eqref{def-ft} with $x \in C$ for computation
of the distance distribution of $C$ as in \cite{Boy95a}.  


\section{Linear programming bounds}\label{s3}

Here we provide slight modifications of the classical linear programming (LP) bounds for spherical codes and designs \cite{DGS,KL}. The
proofs follow directly from the identity
\[ f(1) \cdot |C|+\sum_{x,y \in C, x \neq y} f(\langle x,y \rangle)=
f_0 \cdot |C|^2 + \sum_{i=1}^{\deg{f}} f_iM_i(C), \]
where $f(t)=\sum_{i=0}^{\deg{f}} f_iP_i^{(n)}(t)$ 
and $C \subset \mathbb{S}^{n-1}$ is a code. The
identity itself is obtained by computing both 
sides of $f(t)=\sum_{i=0}^{\deg{f}} f_iP_i^{(n)}(t)$ in all inner products of points $C$ and summing. 

\begin{thm}[LP for unrestricted codes~\cite{DGS,KL}] \label{general-lp}
 Let $T \subset [-1,1)$ and $f(t) \in \mathbb{R}[t]$ be such that

{\rm (A1)} $f(t) \leq 0$ for all $t \in T$;

{\rm (A2)} $f_0>0$ and $f_i \geq 0$ for all $i$ in the Gegenbauer expansion $f(t)=\sum_{i=0}^{\deg{f}} f_iP_i^{(n)}(t)$.

Then 
\[ |C| \leq \frac{f(1)}{f_0} \]
for any spherical code $C \subset \mathbb{S}^{n-1}$ such that $I(C) \subset T$. 
If, in addition, $C$ is a spherical $\tau$-design, then the condition {\rm (A2)} can be replaced by

{\rm (A2-$\tau)$} $f_0>0$ and $f_i \geq 0$ for all $i \geq \tau+1$ in the Gegenbauer expansion of $f$.

\end{thm}

If a spherical code $C \subset \mathbb{S}^{n-1}$ with $I(C) \subset T$ 
and a polynomial $f$ satisfying (A1) and (A2) are such that
$|C|=f(1)/f_0$, then the set of the zeros of $f$ contains $I(C)$ and the
equalities $f_iM_i(C)=0$ follow. In particular, if $f_i>0$ for
$i=1,2,\ldots,m$ for some $m \leq \deg{f}$, then $C$ is a 
spherical $m$-design.

\begin{thm}[LP for antipodal codes~\cite{DGS,KL}] \label{general-antipodal-lp}
 Let $T \subset [-1,1)$ and $f(t) \in \mathbb{R}[t]$ be such that

{\rm (B1)} $f(t) \leq 0$ for all $t \in T$;

{\rm (B2)} $f_0>0$ and $f_i \geq 0$ for all even $i$ in the Gegenbauer expansion $f(t)=\sum_{i=0}^{\deg{f}} f_iP_i^{(n)}(t)$.

Then 
\[ |C| \leq \frac{f(1)}{f_0} \]
for any antipodal spherical code $C \subset \mathbb{S}^{n-1}$ such that $I(C) \subset T$.

If, in addition, $C$ is a spherical $\tau$-design, then the condition {\rm (B2)} can be replaced by

{\rm (B2-$\tau)$} $f_0>0$ and $f_i \geq 0$ for all even $i \geq \tau+1$ in the Gegenbauer expansion of $f$.

\end{thm}

If an antipodal spherical code $C \subset \mathbb{S}^{n-1}$ with 
$I(C) \subset T$ and a polynomial $f$ satisfying (B1) and (B2) are 
such that $|C|=f(1)/f_0$, then (exactly as above) the set of the zeros of $f$ contains $I(C)$ and the equalities $f_iM_i(C)=0$ follow. 
In particular, if $f_i>0$ for $i=2,4,\ldots,m$ for some even $m$, $2 \leq m \leq \deg{f}$, then $C$ is an antipodal spherical $(m+1)$-design. 

\begin{remark}
Theorem \ref{general-antipodal-lp} with conditions
(B1) and (B2) gives, in fact, the linear programming bound for codes in the real projective
space.
\end{remark}

\begin{thm}[LP for spherical designs~\cite{DGS}] \label{general-lp-des}

 Let $\tau$ be positive integer, $T \subset [-1,1)$, and $f(t) \in \mathbb{R}[t]$ be such that

{\rm (C1)} $f(t) \geq 0$ for all $t \in T$;

{\rm (C2)} $f_i \leq 0$ for all $i \geq \tau+1$ in the Gegenbauer expansion $f(t)=\sum_{i=0}^{\deg{f}} f_iP_i^{(n)}(t)$.

Then 
\[ |C| \geq \frac{f(1)}{f_0} \]
for any spherical $\tau$-design $C \subset \mathbb{S}^{n-1}$ such that $I(C) \subset T$. 
\end{thm}

We first note that if $\deg{f} \leq \tau$, the condition (C2) is automatically satisfied. If a spherical $\tau$-design 
$C \subset \mathbb{S}^{n-1}$ with $I(C) \subset T$ and a 
polynomial $f$ satisfying (C1) and (C2) are such that
$|C|=f(1)/f_0$, then (again) the set of the zeros of $f$ contains $I(C)$ and the
equalities $f_iM_i(C)=0$ follow. In particular, if $\tau < \deg{f}$ 
and $f_i<0$ for $i=\tau+1,\ldots,m$ for some $m \geq \tau+1$, 
then $C$ is a spherical $m$-design. 

The conditions for attaining the bounds from
Theorems \ref{general-lp}, \ref{general-antipodal-lp}, and \ref{general-lp-des} 
suggest how good polynomials could be constructed. If we have a good candidate (a code like in dimension 48), it is clear that we need double zeros at the internal points of the set of allowed inner products while simple zeros are possible in the end points with a care for proper sign changes. Then we turn to the strength of the candidate as a spherical design and/or its antipodality (in our case, we have antipodal 11-designs) to complete the construction. 

\section{Kissing codes in 48 dimensions without inner products in \texorpdfstring{$F$}{F} and \texorpdfstring{$F^\prime$}{F'}}

In this section we mainly follow Chapter 7 of the glorious book~\cite{CS}.
\textit{A unimodular lattice} is an integral lattice of determinant $1$ or $-1$. The lattice is called \textit{even} if every vector of the lattice has even norm. Now consider an even unimodular lattice in dimension $n = 24k$; it is known that its minimal vectors have norm at most $2k+2$ and a lattice attaining this bound is called \textit{extremal}.
Recall that there are known four examples of even unimodular extremal lattices, namely $P_{48p}, P_{48q}, P_{48m}$ and $P_{48n}$ 
(the latter lattice was discovered in 2013~\cite{nebe2014fourth}, which is later than the contemporary edition of the book~\cite{CS}).

Modular form theory gives that every even unimodular lattice in 48 dimensions has 52\,416\,000 minimal vectors. Since the difference of any two lattice vectors also belongs to the lattice, the set of minimal vectors form an antipodal kissing code of that cardinality.

\begin{thm}[Venkov,~\cite{venkov1984even}] \label{thm:venkov}
    Let $C$ be a nonempty set of all vectors of an extremal even unimodular lattice of dimension
$n = 24k$ with a fixed norm. Then $C$ is a (properly rescaled) spherical 11-design.
\end{thm}

Furthermore, every spherical code formed by the minimal vectors of an even unimodular extremal lattice in 48 dimensions is distance invariant (since it is a spherical 11-design which
has 8 distinct distances; cf. \cite[Theorem 7.4]{DGS}) 
and it has the following distance distribution:
\begin{align}
\begin{split}~\label{eq:distdist}
A_{-1} &= 1, \\
A_{1/2}=A_{-1/2} &= 36\,848, \\
A_{1/3}=A_{-1/3} &= 1\,678\,887, \\
A_{1/6}=A_{-1/6} &= 12\,608\,784, \\
A_{0}  &=  23\,766\,960.\\
\end{split}
\end{align}
The calculations can be found, in particular, in paper~\cite{ozeki2016numerical}, but they also follow from Theorems~\ref{thm:venkov} and~\ref{11des-48}.

It is worth noting that the set of $A_0=23\,766\,960$ vectors, orthogonal to a fixed one, defines a 47-dimensional kissing configuration. 

\section{Maximal spherical codes with forbidden distances}

In this section we prove that the kissing number in 48 dimensions among
antipodal spherical codes with forbidden inner products in $F$ is 
52\,416\,000.

In what follows we denote
\[ I:=\left[-1,-\frac{1}{3}\right] \cup  \left[-\frac{1}{6},\frac{1}{6}\right] \cup \left[\frac{1}{3}, \frac{1}{2}\right]=
            \left[-1,\frac{1}{2}\right] \setminus F. \] 
and assume that $C \subset \mathbb{S}^{47}$ has a set of inner products $I(C) \subset I$. 
Then Theorem \ref{general-lp} or \ref{general-antipodal-lp} implies that any polynomial $f$ which satisfies either (A1) and (A2) (or (A2-$\tau$), respectively) or (B1) and (B2) (or (B2-$\tau$), respectively), provides 
a upper bound $|C| \leq f(1)/f_0$ for the corresponding $C$ (antipodal or a $\tau$-design in the case of (A2-$\tau$) or (B2-$\tau$)). We present
such a polynomial in the next theorem.

\begin{thm} \label{kiss-48}
Let $C \subset \mathbb{S}^{47}$ be a spherical code with $I(C) \subset I$ such that $C$ is antipodal or $C$ is a 3-design.
Then $|C| \leq 52\,416\,000$. If the equality is attained, then $C$ is an antipodal spherical 11-design and, moreover, it is distance invariant and 
its distance distribution is as given in~\eqref{eq:distdist}.
\end{thm} 

\begin{proof}
According to the remarks in the end of Section \ref{s3} we search for a polynomial with double zeros at $0$ and $-1/2$ and simple zeros at $\pm 1/3$, $\pm 1/6$ and $1/2$. A simple zero at $-1$ does not work 
and we go for a double (moreover, degree 11 is allowed by the candidate, which is an 11-design).

Thus, we consider the polynomial
\begin{eqnarray*} h(t) &:=& (t+1)^2t^2\left(t+\frac{1}{2}\right)^2\left(t^2-\frac{1}{36}\right)\left(t^2-\frac{1}{9}\right)
                               \left(t-\frac{1}{2}\right) \\
&=&t^{11}+\frac{5t^{10}}{2}+\frac{29t^9}{18}-\frac{17t^8}{36}-\frac{959t^7}{1296}-\frac{259t^6}{2592}+\frac{97t^5}{1296}+\frac{11t^4}{648}-\frac{t^3}{648}-\frac{t^2}{2592}.
\end{eqnarray*}
By this definition, we have
\[ h(t) \leq 0 \mbox{ for } t \in I, \] 
so $h(t) \leq 0$ for any $t \in I(C)$ and 
the conditions (A1) of Theorem \ref{general-lp} and (B1) of Theorem \ref{general-antipodal-lp} are satisfied with $T=I$. 

The Gegenbauer expansion $h(t)=\sum_{i=0}^{11} h_iP_i^{(48)}(t)$ has
\[ h_0=\frac{1}{13478400}, \ h_1=\frac{3961}{1758931200}, \ h_2=\frac{47}{8794656}, \ h_3=-\frac{118957}{ 811814400}, \]
\[ h_4=\frac{122059}{1563494400}, \ h_5=\frac{376856011}{32716120320}, \ h_6=\frac{231656467}{3008378880}, \ h_7=\frac{399983395}{1342199808}, \] \[ h_8=\frac{439011349}{577290240}, \ h_9=\frac{3260719}{2589120}, \ h_{10}=\frac{16303595}{14729216}, \ h_{11}=\frac{2075003}{5523456}. \]
By chance or not, $h(1)/h_0=52\,416\,000$. The only negative 
coefficient is $h_3$ which does not cause problems under our
assumptions (antipodality or the 3-design property).

If $C$ is antipodal, then we apply Theorem \ref{general-antipodal-lp} with the polynomial $h(t)$, the code $C$, and $T=I$. As said above, condition (B1) is satisfied. Inspecting the coefficients $h_i$ as given above, we see that condition (B2) is also satisfied. 
Therefore 
\[
|C| \leq \frac{h(1)}{h_0} = 52\,416\,000.
\]

If $C$ is a 3-design, then we apply Theorem \ref{general-lp} with the polynomial $h(t)$, the code $C$, and $T=I$. Clearly conditions~(A1) and (A2-$3$) hold, so again
\[
|C| \leq \frac{h(1)}{h_0} = 52\,416\,000.
\]
It also follows that any attaining 3-design $C \subset \mathbb{S}^{47}$ 
is, in fact, a spherical 11-design and all possible inner products of $C$ are 
$-1, \pm 1/2, \pm 1/3$, and $\pm 1/6$. This implies that $C$ is distance invariant \cite[Theorem 7.4]{DGS}.
Resolving the Vandermonde system which is obtained 
from \eqref{def-ft} for polynomials $f(t)=1,t,t^2\ldots,t^{8}$ (similarly to 
the proof of Theorem 7.4 in \cite{DGS}; see also \cite[Theorem 2.1]{Boy95a}), we obtain\footnote{In fact, three more equations are available but they are consequences of the first nine.} the distance distribution~\eqref{eq:distdist}. 
The so found distance distribution shows that $C$ is antipodal. 

In the computation of the distance distribution in the case of antipodal $C$ we have that $A_{-1}=1$, $A_{1/2}=A_{-1/2}$,  $A_{1/3}=A_{-1/3}$, and $A_{1/6}=A_{-1/6}$. Thus it is enough
to apply~\eqref{def-ft} with $f(t)=t^{2k}$, $k=0,1,2,3,4$. The result is, of course, the same.  
\end{proof}

\section{Minimal spherical 11-designs with forbidden distances}

In this section we prove that the kissing configurations of 52\,416\,000 points on 
$\mathbb{S}^{n-1}$ are also minimal spherical 11-designs if the forbidden 
intervals are either $F$ of $F^\prime$. 

\begin{thm} \label{11des-48}
Let $C \subset \mathbb{S}^{47}$ be a spherical $11$-design with $I(C) \subset [-1,1) \setminus F$. Then
$|C| \geq 52\,416\,000$. If the equality is attained, then 
$C$ is distance invariant and its distance distribution is as given in~\eqref{eq:distdist}.
\end{thm} 

\begin{proof} 
Let us consider the polynomial
\[ g(t)=(t+1)t^2\left(t+\frac{1}{2}\right)^2\left(t^2-\frac{1}{36}\right)\left(t^2-\frac{1}{9}\right)\left(t-\frac{1}{2}\right)^2. \]
The motivation of the choice of this polynomial is similar (and even simpler since the condition (C2) of Theorem \ref{general-lp-des} is satisfied by the choice of degree 11) to the argument for $h$ in Theorem \ref{kiss-48}.

Since $g(t) \geq 0$ for $t \in I$ and $\deg{g}=11$, it 
follows from Theorem \ref{general-lp-des} with
$T=[-1,1) \setminus F$
that 
\[ |C| \geq \frac{g(1)}{g_0}=52\,416\,000\]
(we note that $g_0=1/53913600$ and $g(1)=35/36$)
for any spherical 11-design $C \subset \mathbb{S}^{n-1}$ such that $I(C) \subset T$. 

The distance distribution of $C$ is computed in the same 
way as in Theorem \ref{kiss-48}.
\end{proof}

\begin{thm} \label{11des-48-2}
Let $C \subset \mathbb{S}^{47}$ be a spherical $11$-design with $I(C) \subset [-1,1) \setminus F^\prime$. Then
$|C| \geq 52\,416\,000$. If the equality is attained, then 
$C$ is distance invariant and its distance distribution is as given in~\eqref{eq:distdist}.
\end{thm} 

\begin{proof}
Similar to Theorem \ref{11des-48} with
$I(C) \subset [-1,1) \setminus F^\prime$
and the polynomial
\[
u(t)=(t+1)t^2\left(t^2-\frac{1}{36}\right)^2\left(t^2-\frac{1}{9}\right)\left(t^2-\frac{1}{4}\right). 
\]
We have $u(1) = 1225/972$ and $u_0 = 7/291133440$.
\end{proof}

\section{Other problems for codes with forbidden distances}

The technique from this paper can be probably applied in many cases where good spherical codes and designs are known. Sharp spherical codes 
(spherical $(2k-1)$-designs with $k$ distances) attain the general LP
(the Levenshtein bound \cite{Lev79,Lev98}) but 
many other codes with high symmetry do not and, therefore, could be good candidates to be optimal among codes with some appropriately forbidden distances. 

Interesting problems for spherical codes
and designs with forbidden distances come also from the 
energy minimization idea: find code(s) with minimum $h$-energy
among codes with given dimension, cardinality and set 
of forbidden distances. In could be probably interesting to consider min-max and max-min
polarization problems for such codes. We refer to the book \cite{BHS} for a nice exposition of energy and polarization problems.

\bigskip

\textbf{Acknowledgements.} The research of both authors is supported by Bulgarian NSF grant KP-06-N72/6-2023.	

\bibliographystyle{plain}
\bibliography{main}

\begin{thebibliography}{10}

\bibitem{BV08}
Christine Bachoc and Frank Vallentin.
\newblock New upper bounds for kissing numbers from semidefinite programming.
\newblock {\em Journal of the American Mathematical Society}, 21(3):909--924, 2008.

\bibitem{bannai2009survey}
Eiichi Bannai and Etsuko Bannai.
\newblock A survey on spherical designs and algebraic combinatorics on spheres.
\newblock {\em European Journal of Combinatorics}, 30(6):1392--1425, 2009.

\bibitem{BHS}
Sergiy~V. Borodachov, Douglas~P. Hardin, and Edward~B. Saff.
\newblock {\em Discrete energy on rectifiable sets}.
\newblock Springer, 2019.

\bibitem{Boy95a}
Peter Boyvalenkov.
\newblock Computing distance distributions of spherical designs.
\newblock {\em Linear Algebra and its Applications}, 226:277--286, 1995.

\bibitem{Boy95}
Peter Boyvalenkov.
\newblock Extremal polynomials for obtaining bounds for spherical codes and designs.
\newblock {\em Discrete \& Computational Geometry}, 14(2):167--183, 1995.

\bibitem{BBD99}
Peter Boyvalenkov, Silvia Boumova, and Danyo Danev.
\newblock Necessary conditions for existence of some designs in polynomial metric spaces.
\newblock {\em European Journal of Combinatorics}, 20(3):213--225, 1999.

\bibitem{BDM12}
Peter Boyvalenkov, Stefan Dodunekov, and Oleg Musin.
\newblock A survey on the kissing numbers.
\newblock {\em Serdica Math. J}, 38:507--522, 2012.

\bibitem{CKMRV}
Henry Cohn, Abhinav Kumar, Stephen Miller, Danylo Radchenko, and Maryna Viazovska.
\newblock The sphere packing problem in dimension 24.
\newblock {\em Annals of Mathematics}, 185(3):1017--1033, 2017.

\bibitem{CS}
John~H. Conway and Neil~J.~A. Sloane.
\newblock {\em Sphere Packings, Lattices and Groups}.
\newblock Springer-Verlag, New York, Third Edition, 1999.

\bibitem{DGS}
Phillippe Delsarte, Jean-Marie Goethals, and Johan~J. Seidel.
\newblock Spherical codes and designs.
\newblock {\em Geometriae Dedicata}, 6(3):363--388, 1977.

\bibitem{FL}
G{\'a}bor Fazekas and Vladimir~I. Levenshtein.
\newblock On upper bounds for code distance and covering radius of designs in polynomial metric spaces.
\newblock {\em Journal of Combinatorial Theory, Series A}, 70(2):267--288, 1995.

\bibitem{gonccalves2023sphere}
Felipe Gon{\c{c}}alves and Guilherme Vedana.
\newblock Sphere packings in {E}uclidean space with forbidden distances.
\newblock {\em arXiv preprint arXiv:2308.03925}, 2023.

\bibitem{KL}
Grigory~A. Kabatiansky and Vladimir~I. Levenshtein.
\newblock On bounds for packings on a sphere and in space.
\newblock {\em Problemy Peredachi Informatsii}, 14(1):3--25, 1978.

\bibitem{LL22}
Nando Leijenhorst and David de~Laat.
\newblock Solving clustered low-rank semidefinite programs arising from polynomial optimization.
\newblock {\em arXiv preprint arXiv:2202.12077}, 2022.

\bibitem{Lev79}
Vladimir~I. Levenshtein.
\newblock On bounds for packings in $n$-dimensional {E}uclidean space.
\newblock {\em Doklady Akademii Nauk}, 245(6):1299--1303, 1979.

\bibitem{Lev98}
Vladimir~I. Levenshtein.
\newblock Universal bounds for codes and designs.
\newblock In {\em {H}andbook of {C}oding {T}heory}, pages 499--648. Elsevier Amsterdam, 1998.

\bibitem{Mus03}
Oleg~R. Musin.
\newblock The kissing number in four dimensions.
\newblock {\em Annals of Mathematics}, pages 1--32, 2008.

\bibitem{nebe2014fourth}
Gabriele Nebe.
\newblock A fourth extremal even unimodular lattice of dimension 48.
\newblock {\em Discrete Mathematics}, 331:133--136, 2014.

\bibitem{OS79}
Andrew~M. Odlyzko and Neil J.~A. Sloane.
\newblock New bounds on the number of unit spheres that can touch a unit sphere in $n$ dimensions.
\newblock {\em Journal of Combinatorial Theory, Series A}, 26(2):210--214, 1979.

\bibitem{ozeki2016numerical}
Michio Ozeki.
\newblock A numerical study of {S}iegel theta series of various degrees for the 48-dimensional even unimodular extremal lattices.
\newblock {\em Tsukuba Journal of Mathematics}, 40(2):139--186, 2016.

\bibitem{Sch42}
Isaak~J. Schoenberg.
\newblock Positive definite functions on spheres.
\newblock {\em Duke Math. J.}, 9(1):96--108, 1942.

\bibitem{Szego1975}
G{\'a}bor Szeg{\H o}.
\newblock {\em Orthogonal Polynomials}.
\newblock AMS, Providence, RI, 1975.

\bibitem{venkov1984even}
Boris~B. Venkov.
\newblock Even unimodular extremal lattices.
\newblock {\em Trudy Mat. Inst. Steklov.}, 165:43, 1984.

\bibitem{Via17}
Maryna~S. Viazovska.
\newblock The sphere packing problem in dimension 8.
\newblock {\em Annals of Mathematics}, pages 991--1015, 2017.

\bibitem{Y97}
Vladimir~A. Yudin.
\newblock Lower bounds for spherical designs.
\newblock {\em Izvestiya: Mathematics}, 61(3):673--683, 1997.

\end{thebibliography}

\end{document}